\documentclass[11pt]{article}

\usepackage[T1]{fontenc}
\usepackage[utf8]{inputenc}
\usepackage{lmodern}
\usepackage[margin=1in]{geometry}
\usepackage{amsmath,amssymb,amsthm,mathtools}
\usepackage{microtype}
\usepackage{booktabs,array}
\usepackage{needspace}
\usepackage[numbers,sort&compress]{natbib}
\usepackage[hidelinks]{hyperref}
\usepackage{url}
\hypersetup{
  pdftitle={Sharp Bounds for Graph-Codes Avoiding Complete Bipartite Differences},
  pdfsubject={Extremal densities and linear constructions for complete bipartite graph-codes},
  pdfkeywords={graph-codes, complete bipartite graphs, BCH codes, forbidden intersections}
}

\newtheorem{theorem}{Theorem}[section]
\newtheorem{lemma}[theorem]{Lemma}
\newtheorem{proposition}[theorem]{Proposition}
\newtheorem{corollary}[theorem]{Corollary}
\theoremstyle{definition}

\newcommand{\F}{\mathbb F}
\newcommand{\G}{\mathcal G}
\newcommand{\C}{\mathcal C}
\newcommand{\E}{\mathbb E}
\newcommand{\rank}{\operatorname{rank}}
\newcommand{\lin}{\mathrm{lin}}
\numberwithin{equation}{section}
\title{Sharp Bounds for Graph-Codes Avoiding\\Complete Bipartite Differences}
\author{{\bf Nader H. Bshouty}\\ Faculty of Computer Science\\  Technion Israel Institute of Technology\\ bshouty@cs.technion.ac.il}
\date{}

\begin{document}
\maketitle
\vspace{-1.5em}

\begin{abstract}
Let $d_{K_{s,t}}(n)$ be the largest fraction of all graphs on
$[n]$ in a family containing no two distinct members whose
symmetric difference is a copy of $K_{s,t}$. We determine
this maximum density for every fixed pair of positive integers
$s,t$. If both part sizes are odd, it is known that the density is exactly
$1/2$. Otherwise, let $2h$ be the largest even
part size. We prove
\[
d_{K_{s,t}}(n)=\Theta_{s,t}(n^{-h}).
\]
For $h>1$, this improves the general $O_{s,t}(1/n)$ upper bound
obtained from Alon's two-copy gluing theorem~\cite{Alon2024}.
The sharp bounds for even stars were already proved by Alon
and are recovered as special cases.

The upper bound uses the Frankl--F\"uredi
forbidden-intersection theorem. For the lower bound,
we construct binary linear codes using BCH codes.
When exactly one part size is even, we assign to
each edge the sum of the binary vectors assigned
to its endpoints. When both part sizes are even,
we identify these vectors with elements of a
finite field and assign to each edge the product
of its endpoint elements. In both cases, every
copy of $K_{s,t}$ has a nonzero syndrome.

\end{abstract}

\section{Introduction}

In classical binary error-correcting codes, any two distinct codewords must differ in sufficiently many coordinates. In graph-codes, each coordinate represents a possible edge of a graph. Instead of requiring a minimum number of differences, we forbid certain graphs from appearing as the difference between two codewords. In this paper, we study families of graphs in which the symmetric difference of any two distinct members is not exactly a complete bipartite graph $K_{s,t}$. Our goal is to determine the maximum possible size of such a family.

All graphs in this paper are finite, simple, and undirected. Write
$[n]=\{1,\ldots,n\}$, and let $\G_n$ be the set of all graphs with vertex
set $[n]$. There are $2^{\binom n2}$ such graphs. The \emph{symmetric
difference} $G\triangle G'$ is the graph whose edges belong to exactly
one of $G$ and $G'$. For disjoint nonempty sets $A,B\subseteq[n]$, write
$K_{A,B}$ for the graph whose edge set consists of all pairs with one
endpoint in $A$ and the other in $B$. In particular, vertices outside
$A\cup B$ are isolated in $K_{A,B}$.

Fix positive integers $s,t$. A \emph{$K_{s,t}$-code} is a family
$\C\subseteq\G_n$ such that
\[
 G\triangle G'\ne K_{A,B}
\]
for every two distinct $G,G'\in\C$ and every pair of disjoint sets
$A,B\subseteq[n]$ with $|A|=s$ and $|B|=t$. We study the maximum density
\[
 d_{K_{s,t}}(n)
 =\max\left\{\frac{|\C|}{2^{\binom n2}}:
              \C\text{ is a }K_{s,t}\text{-code}\right\}.
\]
This is Alon's graph-code density parameter~\cite{Alon2024},
specialized to complete bipartite forbidden differences.

Identifying a graph with its binary edge-indicator vector identifies
$\G_n$ with the vector space $\F_2^{\binom n2}$. Vector addition is then
symmetric difference. A graph-code is \emph{binary linear} if it is a
subspace of this vector space. Let $d^{\lin}_{K_{s,t}}(n)$ denote the
maximum density among binary linear $K_{s,t}$-codes. Since every linear
code is also a code without the linearity restriction,
\[
 d^{\lin}_{K_{s,t}}(n)\le d_{K_{s,t}}(n).
\]

\subsection{Main result}

The maximum density $d_{K_{s,t}}(n)$ depends on whether $s$ and $t$ are odd or even. When both are odd, the maximum density is exactly $1/2$~\cite{Alon2024,Versteegen2025}. Otherwise, the density decreases polynomially with $n$. In this case, the exponent of $n$ is half the largest even part size. In particular, the size of an odd part does not affect the exponent.

\begin{theorem}[Complete bipartite classification]\label{thm:main}
Fix positive integers $s,t$. If both $s$ and $t$ are odd,
then
\[
 d^{\lin}_{K_{s,t}}(n)=d_{K_{s,t}}(n)=\frac12.
\]
If at least one of $s,t$ is even, let $2h$ be the largest even integer
among $s,t$. There are positive constants $c_{s,t},C_{s,t}$ such that,
for all sufficiently large $n$,
\[
 c_{s,t}n^{-h}
 \le d^{\lin}_{K_{s,t}}(n)
 \le d_{K_{s,t}}(n)
 \le C_{s,t}n^{-h}.
\]
Equivalently, the classification for arbitrary graph-codes is
\[
 d_{K_{s,t}}(n)=
 \begin{cases}
  \dfrac12,
       & s,t\text{ both odd},\\[1.5mm]
  \Theta_{s,t}(n^{-s/2}),
       & s\text{ even and }t\text{ odd},\\[1.5mm]
  \Theta_{s,t}(n^{-t/2}),
       & s\text{ odd and }t\text{ even},\\[1.5mm]
  \Theta_{s,t}(n^{-\max\{s,t\}/2}),
       & s,t\text{ both even}.
 \end{cases}
\]
\end{theorem}

All asymptotic statements take $n\to\infty$ with $s,t$ fixed.
If $n<s+t$, no copy of $K_{s,t}$ exists on $[n]$, so both densities are
one. We therefore focus on $n\ge s+t$.

The theorem also determines the number of independent binary parity
checks needed by an optimal linear code, up to an additive constant.
The \emph{codimension} of a subspace $\C\subseteq\G_n$ is
$\binom n2-\dim_{\F_2}\C$. A subspace of codimension $r$ has density
$2^{-r}$. Let $r_{s,t}(n)$ be the minimum codimension of a binary linear
$K_{s,t}$-code.

\begin{corollary}[Optimal linear codimension]\label{cor:codimension}
Fix positive integers $s,t$. If both part sizes
are odd, then $r_{s,t}(n)=1$. Otherwise, let $2h$ be the largest even
part size. Then
\[
 r_{s,t}(n)\le h\left\lceil\log_2 n\right\rceil,
\]
and, as $n\to\infty$,
\[
 r_{s,t}(n)=h\log_2 n+O_{s,t}(1).
\]
\end{corollary}

The first bound in Corollary~\ref{cor:codimension} holds for every
$n\ge s+t$. The lower bound implicit in the
second statement follows from the upper bound for arbitrary codes in
Theorem~\ref{thm:main}.

\subsection{Previous work and our results}

The study of graph-codes was developed by Alon, Gujgiczer,
K\"orner, Milojevi\'c, and Simonyi~\cite{AlonEtAl2023}.
They studied families of graphs whose pairwise symmetric
differences satisfy prescribed properties. Alon~\cite{Alon2024}
then considered the problem of forbidding a fixed graph
$H$ as the symmetric difference of two distinct members.
He determined the maximum density, up to constant factors,
when $H$ is a star or a matching with an even number of edges.
In particular, he proved
\[
d_{K_{1,2h}}(n)=\Theta_h(n^{-h}).
\]
His proof uses the forbidden-intersection theorem of
Frankl and F\"uredi~\cite{FranklFuredi1985} for the
upper bound and BCH codes for the lower bound.

Alon's Proposition~1.5 in~\cite{Alon2024} gives an upper bound for a larger
class of graphs. To construct such a graph $H$, start with
two copies of a graph $F$ and choose an independent set
$R$ of vertices in $F$. Merge the two copies of each vertex
in $R$, leaving all other vertices distinct. Alon proved
that the resulting graph satisfies
\[
d_H(n)\le \frac{|V(H)|}{n}.
\]
Every $K_{u,2h}$ can be constructed in this way. Take two
copies of $K_{u,h}$ and merge the corresponding vertices
in their parts of size $u$. The resulting graph is
$K_{u,2h}$. Consequently,
\[
d_{K_{u,2h}}(n)\le\frac{u+2h}{n}=O_{u,h}(n^{-1}).
\]

Conlon, Lee, and Versteegen~\cite[Theorem~1.7]{ConlonLeeVersteegen} extended this upper bound to a larger class of graphs, called \emph{positive graphs}. Their theorem gives
\[
d_H(n)=O_H(n^{-1})
\]
for every positive graph $H$. In particular, it applies
to every $K_{s,t}$ with at least one even part.
Thus, their theorem applies to more graphs than Alon's
Proposition~1.5, but gives the same $O_{s,t}(n^{-1})$
bound for the complete bipartite graphs considered here.

Versteegen~\cite{Versteegen2025} studied the corresponding
problem for linear graph-codes. His Theorem~1.2 shows
that every fixed graph $H$ with an even number of edges
satisfies
\[
d_H^{\lin}(n)=O_H(1/\log n).
\]
His Theorem~1.3 gives a polynomial bound for graphs
that admit an even decomposition\footnote{An even decomposition of $H$ is a sequence
$V(H)=V_0\supset V_1\supset\cdots\supset V_k=\varnothing$
such that, for each $i$, the set
$V_{i-1}\setminus V_i$ is independent, and the number
of edges between $V_{i-1}\setminus V_i$ and $V_i$
is even.}. In particular,
every bipartite graph with an even number of edges
belongs to this class. The theorem therefore implies
\[
d_{K_{s,t}}^{\lin}(n)
=O_{s,t}\bigl(n^{-1/(s+t-2)}\bigr)
\]
whenever $st$ is even. 
For the complete bipartite graphs considered here,
however, the $O_{s,t}(n^{-1})$ upper bound for arbitrary
codes is stronger. Janzer and Yip~\cite{JanzerYip2025}
subsequently strengthened the results on the proportion
of graphs admitting an even decomposition. Their work
does not improve the density bound for a fixed
$K_{s,t}$.

Table~\ref{tab:comparison} summarizes the bounds. 

\begin{table}[htbp]
\centering
\small
\renewcommand{\arraystretch}{1.4}
\begin{tabular}{@{}p{0.24\textwidth}
                    p{0.43\textwidth}
                    p{0.25\textwidth}@{}}
\toprule
Forbidden graph &
Earlier bounds and constructions &
This paper \\
\midrule
$K_{2a+1,2b+1}, a,b\ge 0$
&
$ 1/2$~\cite{Alon2024,Versteegen2025}
&
$1/2$; recovered
\\
\addlinespace

$K_{1,2a}$, $a\ge1$
&
$\Theta_a(n^{-a})$,
Alon~\cite[Theorem~1.2]{Alon2024}
&
$\Theta_a(n^{-a})$;
recovered
\\
\addlinespace

$K_{2a,2b+1}$,
$a,b\ge1$
&
Upper: $O_{a,b}(n^{-1})$
~\cite{Alon2024,ConlonLeeVersteegen}.
&
$\Theta_{a,b}(n^{-a})$
\\
\addlinespace

$K_{2a,2b}$,
$a,b\ge1$
&
Upper: $O_{a,b}(n^{-1})$
~\cite{Alon2024,ConlonLeeVersteegen}.
&
$\Theta_{a,b}
(n^{-\max\{a,b\}})$
\\
\bottomrule
\end{tabular}
\caption{Comparison with earlier results and constructions.
The upper bounds apply to arbitrary graph-codes. The
lower bounds are achieved by linear codes.}
\label{tab:comparison}
\end{table}

\subsection{Our contribution}

Our main result determines the maximum density of
$K_{s,t}$-codes up to constant factors for every
fixed pair of positive integers $s,t$.

For the upper bound, we use essentially the same
proof that Alon~\cite{Alon2024} used for even stars.
In Alon's proof, one vertex is fixed, and the
Frankl--F\"uredi theorem is applied to stars whose
other endpoints form sets of size $2h-1$.
We make the following modification: instead of fixing
one vertex, we fix a set of $u$ vertices and
consider complete bipartite graphs $K_{R,A}$,
where $|R|=u$ and $|A|=2h-1$.
This gives
\[
d_{K_{u,2h}}(n)=O_{u,h}(n^{-h}),
\]
improving the previously known general
$O_{u,h}(n^{-1})$ bound when $h>1$.

For the lower bound, we distinguish two cases.
When exactly one part size is even, the matching
lower bound follows directly from the BCH method
used by Alon~\cite{Alon2024}. 

When both part sizes are even, the BCH construction
based on vertex degrees does not work, since every
vertex of the forbidden graph has even degree.
We instead assign elements of a finite field to
the vertices and assign to each edge the product
of its endpoint elements. The vertex elements
are chosen using BCH codes so that the sums over
the two parts of every forbidden complete
bipartite graph are nonzero. This gives
\[
d_{K_{2a,2b}}^{\lin}(n)
=\Omega_{a,b}(n^{-\max\{a,b\}}),
\]
which matches our upper bound.

\subsection{Proof overview}

We first explain the idea of the upper bound. Suppose that at least one of $s,t$ is even. Write the forbidden graph as $K_{u,2h}$, where $2h$ is an even part size and $u$ is the size of the other part. Fix a set $R\subseteq[n]$ of size $u$. For every set $A\subseteq[n]\setminus R$ of size $2h-1$, consider the complete bipartite graph $K_{R,A}$, whose two parts are $R$ and $A$.

The key observation is that if two such sets $A$ and $B$ satisfy $|A\cap B|=h-1$, then their symmetric difference $A\triangle B$ has exactly $2h$ vertices. Consequently,
\[
K_{R,A}\triangle K_{R,B}
=K_{R,A\triangle B},
\]
which is a copy of the forbidden graph $K_{u,2h}$. Therefore, a $K_{u,2h}$-code cannot contain both $K_{R,A}$ and $K_{R,B}$.

We use the Frankl--F\"uredi theorem to bound how many graphs from this special family can belong to a code. We then show that this bound also gives an upper bound on the density of the entire code. The argument uses the fact that taking the symmetric difference of every graph in a code with the same fixed graph preserves the forbidden-difference condition. This gives $O_{u,h}(n^{-h})$. When both $s$ and $t$ are even, we choose $2h$ to be the larger part size to obtain the stronger bound.

For the lower bound, we construct a linear code by assigning
a binary vector to each edge of the complete graph on $[n]$.
The vectors are chosen so that, for every copy of
$K_{s,t}$, the sum of the vectors assigned to its edges
is nonzero. All vector additions are performed over
$\mathbb F_2$.

Suppose first that the forbidden graph is $K_{u,2h}$,
where $u$ is odd. We begin by assigning a binary vector 
$b_i\in\mathbb F_2^r$ where
$r=h\lceil\log_2 n\rceil$ to each
vertex $i\in[n]$. These vectors are constructed using
BCH codes. They have the property that the sum of the
vectors assigned to any positive even number of at most
$2h$ distinct vertices is nonzero.

 We then assign the vector $b_i+b_j$ to
edge $(i,j)$. Consider a copy of $K_{u,2h}$ with parts
$A$ and $B$, where $|A|=u$ and $|B|=2h$.
When we sum the vectors assigned to its edges, each
vertex vector appears as many times as the degree of
that vertex. Since vertices in $A$ have even degree
and vertices in $B$ have odd degree, the sum is
$\sum_{j\in B}b_j$.
This sum is nonzero by the BCH property.

If both parts of $K_{s,t}$ are even, the preceding
construction does not work because every vertex has
even degree. We therefore use a different assignment
of vectors to edges. We identify the binary vectors
$b_i$ with elements $z_i$ of a finite field of
characteristic two, preserving addition. We then
assign the field element $z_i z_j$ to edge $(i,j)$.

For a complete bipartite graph with parts $A$ and $B$,
the sum of the elements assigned to its edges is
\[
\sum_{i\in A}\sum_{j\in B}z_i z_j
=
\left(\sum_{i\in A}z_i\right)
\left(\sum_{j\in B}z_j\right).
\]
Both factors are nonzero by the BCH property.
Their product is therefore nonzero.

In both cases, we obtain a linear map that assigns
to each graph the sum of the vectors assigned to its
edges. Every copy of $K_{s,t}$ has a nonzero sum.
Consequently, the set of graphs whose sum is zero
forms a linear $K_{s,t}$-code.
\section{Preliminaries}\label{sec:tools}

\subsection{Translations and averaging}
For $X\in\G_n$ and $\C\subseteq\G_n$, define
\[
 \C\triangle X=\{G\triangle X:G\in\C\}.
\]
This is called a {\it translate} of $\C$. If $\C$ is a $K_{s,t}$-code, every
translate is also a $K_{s,t}$-code, since
\[
 (G\triangle X)\triangle(G'\triangle X)=G\triangle G'.
\]

The following standard averaging lemma allows us to
obtain an upper bound on the density of a code by
bounding its intersection with a fixed family of
graphs. It is the translation-averaging form of the
vertex-transitivity argument used by
Alon~\cite[Section~2.1]{Alon2024}. We include a proof for completeness.

\begin{lemma}[Translation averaging, \cite{Alon2024}]\label{lem:averaging}
Let $\mathcal T$ be a nonempty subset of $\G_n$. Suppose that every
$K_{s,t}$-code $\mathcal D\subseteq\G_n$ satisfies
$|\mathcal D\cap\mathcal T|\le M$. Then
\[
 d_{K_{s,t}}(n)\le\frac{M}{|\mathcal T|}.
\]
\end{lemma}

\begin{proof}
Fix a $K_{s,t}$-code $\C$, and choose $X$ uniformly from $\G_n$. For each
fixed $T\in\mathcal T$, the graph $T\triangle X$ is uniform in $\G_n$.
Thus
\[
 \Pr[T\in\C\triangle X]=\frac{|\C|}{2^{\binom n2}}.
\]
Summing these probabilities over $T\in\mathcal T$ gives
\[
 \E\bigl[|(\C\triangle X)\cap\mathcal T|\bigr]
 =|\mathcal T|\frac{|\C|}{2^{\binom n2}}\le|\mathcal T| d_{K_{s,t}}(n).
\]
Every translate $\C\triangle X$ is a $K_{s,t}$-code, so the hypothesis
bounds the random variable inside the expectation by $M$. Dividing by
$|\mathcal T|$ bounds the density of $\C$. 
\end{proof}

\subsection{A forbidden-intersection theorem}

In the proof of the upper bound, we need to bound the size of a family of sets, each containing exactly $2h-1$ elements, such that no two distinct sets have exactly $h-1$ elements in common. We use the following special case of the Frankl--F\"uredi theorem~\cite[Theorem~2.1]{FranklFuredi1985}.
The same
special case is used in Alon's star proof~\cite[Section~2.1]{Alon2024}.

For a finite set $U$, the notation $\binom Uq$ denotes the family of
its $q$-element subsets.

\begin{theorem}[Frankl--F\"uredi, specialized form~\cite{FranklFuredi1985}]\label{thm:FF}
For each integer $h\ge1$, there is a constant $C_h>0$ with the following
property. Let $U$ be a set of size $m\ge2h-1$, and let
$\mathcal A\subseteq\binom U{2h-1}$. If
\[
 |A\cap B|\ne h-1
 \qquad\text{for every distinct }A,B\in\mathcal A,
\]
then
\[
 |\mathcal A|\le C_h m^{h-1}.
\]
For $h=1$, one may take $C_1=1$.
\end{theorem}

We explain how this special case follows from the
Frankl--F\"uredi theorem~\cite{FranklFuredi1985},
using the formulation given in
Alon~\cite[Theorem~2.2]{Alon2024}.
For $h\ge2$, apply that theorem with the parameters
$\ell=2h-1$ and $\ell_1=\ell_2=h-1$.
The theorem applies to families of sets of size
$\ell$ in which the intersection of any two distinct
sets has size either at least $\ell-\ell_1$
or strictly less than $\ell_2$.

With our choice of parameters, these conditions become
$|A\cap B|\ge h$
or
$|A\cap B|<h-1$.
Since the intersection size is an integer, this is
equivalent to requiring that
$|A\cap B|\ne h-1$.
The Frankl--F\"uredi theorem therefore gives
\[
|\mathcal A|
=O_h\bigl(m^{\max\{\ell_1,\ell_2\}}\bigr)
=O_h(m^{h-1}).
\]

For $h=1$, every set in $\mathcal A$ has exactly
one element. Any two distinct such sets have
intersection size zero, which is forbidden because
$h-1=0$. Therefore, $\mathcal A$ contains at most
one set, as required.

\subsection{Linear parity checks}

The lower bounds will be obtained by assigning a fixed binary vector
to each possible edge. The sum of the labels on the edges of a graph
is its \emph{syndrome}. Fix labels $c_{ij}\in\F_2^r$ for
$1\le i<j\le n$, and define
\[
 \sigma(G)=\sum_{(i,j)\in E(G)}c_{ij}.
\]
Here $E(G)$ is the edge set of $G$, and the sum is over $\F_2$. An edge
that occurs in both $G$ and $G'$ contributes twice and cancels, so
$\sigma(G\triangle G')=\sigma(G)+\sigma(G')$. In particular, $\sigma$
is linear in the edge-indicator vector of $G$.

The following lemma is the syndrome construction
used by Alon~\cite[Section~2.2]{Alon2024},
specialized to complete bipartite graphs.
We include a proof for completeness.
\begin{lemma}[Syndrome construction~\cite{Alon2024}]\label{lem:syndrome}
Suppose the labels satisfy $\sigma(K_{A,B})\ne0$ for every pair of
disjoint sets $A,B\subseteq[n]$ with $|A|=s$ and $|B|=t$. Then
$\ker\sigma$ is a binary linear $K_{s,t}$-code of codimension at most
$r$. Its density is at least $2^{-r}$.
\end{lemma}

\begin{proof}
The kernel is a subspace. If $G,G'\in\ker\sigma$, linearity gives
$\sigma(G\triangle G')=0$. The assumed nonzero syndrome of every
$K_{A,B}$ therefore prevents $G\triangle G'$ from being a forbidden
difference. By rank--nullity, the codimension of the kernel is
$\rank\sigma\le r$. Its density is consequently
$2^{-\rank\sigma}\ge2^{-r}$.
\end{proof}

We will also use Lemma~\ref{lem:syndrome} when the edge labels lie in
$\F_{2^r}$. This causes no change: the additive group of $\F_{2^r}$ is
an $r$-dimensional vector space over $\F_2$. Choosing a binary basis
turns each field-valued syndrome into $r$ binary parity checks.

\section{The upper bound for arbitrary graph-codes}\label{sec:upper}
In this section we prove the upper bounds. 
\subsection{At least one part is even}
We now use Theorem~\ref{thm:FF} to prove an upper
bound on the density of a $K_{u,2h}$-code.
Lemma~\ref{lem:averaging} will turn the resulting intersection bound
into a density bound.

The following proves the upper bound of Theorem~\ref{thm:main}
\begin{proposition}[An even part gives an upper bound]\label{prop:upper}
Fix integers $u,h\ge1$. Then
\[
 d_{K_{u,2h}}(n)=O_{u,h}(n^{-h}).
\]
More precisely, with the constant $C_h$ from Theorem~\ref{thm:FF}, for
$n\ge u+2h$,
\[
 d_{K_{u,2h}}(n)
 \le \frac{C_h (n-u)^{h-1}}{\binom {n-u}{2h-1}}.
\]
\end{proposition}

\begin{proof}
Fix a set $R\subseteq[n]$ of size $u$, put $U=[n]\setminus R$, and let
$m:=|U|=n-u$. Consider
\[
 \mathcal T=\left\{K_{R,A}:A\in\binom U{2h-1}\right\}.
\]
Since $R$ is nonempty, distinct sets $A$ give distinct graphs  $K_{R,A}$, so
$|\mathcal T|=\binom m{2h-1}$.

For any $K_{u,2h}$-code $\mathcal D$, the family
\[
 \mathcal A=\left\{A\in\binom U{2h-1}:K_{R,A}\in\mathcal D\right\}
\]
contains no pair with intersection size $h-1$. Indeed, such a pair
would satisfy $|A\triangle B|=2(2h-1)-2(h-1)=2h$, and hence
\[
 K_{R,A}\triangle K_{R,B}=K_{R,A\triangle B}
\]
would be a forbidden $K_{u,2h}$. By Theorem~\ref{thm:FF},
\[
 |\mathcal D\cap\mathcal T|=|\mathcal A|\le C_hm^{h-1}.
\]
Lemma~\ref{lem:averaging} therefore gives
\[
 d_{K_{u,2h}}(n)
 \le\frac{C_hm^{h-1}}{\binom m{2h-1}}.
\]
For fixed $h$, the denominator has order $m^{2h-1}$. Since
$m=n-u=\Theta_u(n)$, this proves the stated asymptotic bound.
\end{proof}

When both parts are even, we apply Proposition~\ref{prop:upper} with
the larger part as the part of size $2h$. This proves the upper bound
in Theorem~\ref{thm:main} whenever at least one part is even. We next
construct linear codes with matching density.

\subsection{Both parts odd}

When both $s$ and $t$ are odd, the equality
$d_{K_{s,t}}(n)=1/2$ is already known; see
Versteegen~\cite{Versteegen2025} and
Alon~\cite{Alon2026}.
We include a short proof for completeness.

\begin{proposition}[Both parts odd]
\label{prop:odd-upper}
Let $s,t\ge1$ be odd integers. Then, for every
$n\ge s+t$,
\[
d_{K_{s,t}}(n)=\frac12.
\]
\end{proposition}

\begin{proof}
We first prove the upper bound.

Fix two disjoint sets $A,B\subseteq[n]$ with
$|A|=s$ and $|B|=t$, and let $H=K_{A,B}$.
The map
$G\longmapsto G\triangle H$
partitions all graphs on $[n]$ into disjoint pairs
$\{G,G\triangle H\}$. The two graphs in each pair
are distinct, and their symmetric difference is $H$. Therefore, a $K_{s,t}$-code can contain at most one
graph from each pair. Since there are
$2^{\binom n2-1}$ such pairs, every $K_{s,t}$-code
$\mathcal C$ satisfies
$|\mathcal C|\le 2^{\binom n2-1}$.
Consequently,
\[
d_{K_{s,t}}(n)\le\frac12.
\]

For the matching lower bound, consider the family
of all graphs on $[n]$ having an even number of edges.
Since $s$ and $t$ are odd, every copy of $K_{s,t}$
has an odd number $st$ of edges. However, the
symmetric difference of two graphs with an even
number of edges always has an even number of edges.
Thus, this family is a $K_{s,t}$-code.
Exactly half of all graphs on $[n]$ have an even
number of edges. Therefore,
\[
d_{K_{s,t}}(n)\ge\frac12.
\]
Combining the two bounds proves the proposition.
\end{proof}

\section{Linear constructions}\label{sec:lower}

We first assign a binary vector $b_i$ to each vertex
$i\in[n]$. These vectors are chosen so that the sum
of the vectors assigned to any positive even number
of at most $2h$ distinct vertices is nonzero.

We then prove a lemma that allows us to view these
binary vectors as elements of a finite field and
multiply them. Finally, we use these results to
construct linear $K_{s,t}$-codes. We give one
construction when exactly one part is even and
another when both parts are even.

\subsection{Vertex labels that detect even subsets}

In our constructions, we only need binary vectors
$b_1,\ldots,b_n$ such that the sum of the vectors
assigned to any positive even number of at most
$2h$ distinct vertices is nonzero.
We use a modification of the standard BCH construction
\cite[Chapter~5]{HuffmanPless2003}, which was also used
by Alon~\cite[Section~2.2]{Alon2024}.
Since we only require nonzero sums for even-sized
subsets, we can allow zero as one of the field elements.
This reduces the number of binary coordinates from\footnote{This refinement does not affect our main result.
Both bounds give the same asymptotic density
$\Theta_{s,t}(n^{-h})$. The improvement only reduces
the number of binary coordinates by $h$ when $n$
is a power of two.}
$h\lceil\log_2(n+1)\rceil$ to
$h\lceil\log_2 n\rceil$.

\begin{lemma}[Even-subset labels]\label{lem:BCH}
Let $n\ge2$ and $h\ge1$ be integers, and put
$m=\lceil\log_2 n\rceil$. There exist vectors
$b_1,\ldots,b_n\in\F_2^{hm}$ such that
\[
 \sum_{i\in S}b_i\ne0
\]
whenever $S\subseteq[n]$ has positive even size at most $2h$.
\end{lemma}

\begin{proof}
Choose distinct elements $x_1,\ldots,x_n$ of $\F_{2^m}$, allowing
zero, and define
\[
 b_i=(x_i,x_i^3,\ldots,x_i^{2h-1})\in \F_{2^m}^h.
\]
Suppose that an even set $S$ of size $v$, where $2\le v\le2h$, has
zero label sum. For $j\ge1$, put
 $p_j=\sum_{i\in S}x_i^j$.
The assumed zero sum gives $p_j=0$ for every odd $j\le2h-1$.
The identity $p_{2j}=p_j^2$ then gives $p_j=0$ for
$1\le j\le v-1$: write each such $j$ as a power of two times an odd
integer and apply the identity repeatedly. Also, define
$p_0=\sum_{i\in S}1$. Since $v$ is even, $p_0=0$ in $\F_{2^m}$.

Writing $S=\{i_1,\ldots,i_v\}$, these equations say
\[
 \begin{pmatrix}
  1&\cdots&1\\
  x_{i_1}&\cdots&x_{i_v}\\
  \vdots&&\vdots\\
  x_{i_1}^{v-1}&\cdots&x_{i_v}^{v-1}
 \end{pmatrix}
 \begin{pmatrix}1\\ \vdots\\1\end{pmatrix}=0.
\]
The matrix is a Vandermonde matrix with determinant
$\prod_{p<q}(x_{i_q}-x_{i_p})\ne0$, because the $x_i$ are distinct.
It cannot annihilate the nonzero all-one vector, a contradiction.
Expanding the $h$ coordinates in a binary basis of $\F_{2^m}$ gives the
required vectors in $\F_2^{hm}$ and preserves all subset sums.
\end{proof}

\subsection{Constructing linear codes from vertex vectors}

In this subsection, we prove that if binary vectors
assigned to the vertices have nonzero sums over
every subset of size $s$ or $t$, then these vectors
can be used to construct a binary linear
$K_{s,t}$-code.

\begin{lemma}[Linear code construction]
\label{lem:field}
Let $n,s,t,r\ge1$, and suppose that vectors
$w_1,\ldots,w_n\in\F_2^r$ satisfy
\[
 \sum_{i\in S}w_i\ne0
 \qquad\text{whenever }S\subseteq[n]\text{ and }|S|\in\{s,t\}.
\]
Then there is a binary linear $K_{s,t}$-code of codimension at most
$r$ and density at least $2^{-r}$.
\end{lemma}

\begin{proof}
Choose an $\F_2$-linear isomorphism
$\phi:\F_2^r\longrightarrow\F_{2^r}$ and put $z_i=\phi(w_i)$.
Label edge $(i,j)$ by $z_i z_j$, and let $\sigma$ be the sum of these
fixed edge labels. For disjoint $A,B\subseteq[n]$ with $|A|=s$ and
$|B|=t$, distributivity gives
\begin{equation}\label{eq:factorization}
 \sigma(K_{A,B})
 =\left(\sum_{i\in A}z_i\right)
  \left(\sum_{j\in B}z_j\right).
\end{equation}
Both factors are nonzero by the hypothesis. Their product is nonzero
because it is taken in a field. Lemma~\ref{lem:syndrome} therefore
shows that $\ker\sigma$ is the required code.
\end{proof}

Only additivity is required of $\phi$; it need not preserve
coordinatewise multiplication.

\subsection{The linear construction}

We now use Lemma~\ref{lem:BCH} for the largest even part. Depending on
the parity of the other part, we add or multiply endpoint labels.

The following proves the lower bound of Theorem~\ref{thm:main}.

\begin{proposition}[Linear lower bound]\label{prop:linear}
Suppose that at least one of $s,t$ is even, and let $2h$ be the largest
even part size. For every $n\ge s+t$, there is a binary linear
$K_{s,t}$-code of codimension at most~$h\left\lceil\log_2 n\right\rceil$.

In particular,
\[
 d^{\lin}_{K_{s,t}}(n)\ge2^{-h\lceil\log_2 n\rceil}\ge \Omega_{s,t}(n^{-h}).
\]
\end{proposition}

\begin{proof}
Write $K_{s,t}=K_{u,2h}$, put $m=\lceil\log_2 n\rceil$, and take the
vectors from Lemma~\ref{lem:BCH}. Identify them additively with
$z_1,\ldots,z_n\in\F_{2^{hm}}$. This preserves the nonzero sums of
even subsets of size at most~$2h$.

If $u$ is odd, label edge $(i,j)$ by $z_i+z_j$. For disjoint sets $A,B$
with $|A|=u$ and $|B|=2h$, the vertices in $A$ have even degree and
those in $B$ have odd degree. The syndrome of $K_{A,B}$ is therefore
\[
 \sigma(K_{A,B})=\sum_{j\in B}z_j\ne0,
\]
where the inequality follows from Lemma~\ref{lem:BCH}.
Lemma~\ref{lem:syndrome} gives the required code.

If $u$ is even, then $2\le u\le2h$. Lemma~\ref{lem:BCH} guarantees
nonzero subset sums of both sizes $u$ and $2h$. Applying Lemma~\ref{lem:field} with $r=hm$,
we obtain a binary linear $K_{s,t}$-code
of codimension at most $hm$ and density
at least $2^{-hm}$.
\end{proof}

\begingroup
\small
\setlength{\bibsep}{5pt}
\bibliographystyle{plainnat}
\bibliography{reference}
\endgroup

\end{document}